\documentclass[a4paper,11pt]{amsart}
\usepackage{amsmath,inputenc,euscript,amssymb,geometry}
\usepackage{graphicx}
\usepackage{amssymb}
\usepackage{latexsym}
\usepackage{amssymb,amsbsy,amsmath,amsfonts,amssymb,amscd}
\usepackage{hyperref}
\usepackage{xcolor}
\newcommand{\D}{\displaystyle}
\newtheorem{lemma}{Lemma}[section]
\newtheorem{theorem}[lemma]{Theorem}

\newtheorem{remark}[lemma]{Remark}

\begin{document}
\title[]{Singularity formation for the 1-D cubic NLS\\ and the Schr\"odinger map on $\mathbb S^2$}
\author[V. Banica]{Valeria Banica}
\address[V. Banica]{Laboratoire de Math\'ematiques et Mod\'elisation d'\'Evry (UMR 8071)\\D\'epartement de Math\'ematiques\\ Universit\'e d'\'Evry, 23 Bd. de France, 91037 \'Evry\\ France, Valeria.Banica@univ-evry.fr} 

\author[L. Vega]{Luis Vega}
\address[L. Vega]{Departamento de Matem\'aticas, Universidad del Pais Vasco, Aptdo. 644, 48080 Bilbao, Spain, luis.vega@ehu.es,
\hfill\break\indent \and
BCAM Alameda Mazarredo 14, 48009 Bilbao, Spain, lvega@bcamath.org} 
%\date\today

\maketitle
{\it To Vladimir Georgiev, for his 60$^{th}$ birthday.}
\begin{abstract}
In this note we consider the 1-D cubic Schr\"odinger equation with data given as small perturbations of a Dirac-$\delta$ function and some other related equations. We first recall that although the problem for this type of data is ill-posed one can use the geometric framework of the Schr\"odinger map to define the solution beyond the singularity time. Then, we find some natural and well defined geometric quantities that are not regular at time zero. Finally, we make a link between these results and some known phenomena in fluid mechanics that inspired this note.
\end{abstract}

\section{Introduction}

\subsection{Low regularity issues for the 1-D cubic Schr\"odinger equation}
We consider the 1-D cubic Schr\"odinger equation
\begin{equation}\label{NLS0}
\left\{\begin{array}{cc}i\psi_t+\psi_{xx}\pm|\psi|^2\psi=0,\\\psi(0)=\psi_0.\end{array}\right.
\end{equation}

The regularity threshold in Sobolev spaces for the Cauchy problem is $L^2$. More precisely, the equation is known to be well-posed in $H^{s}$, for $s\geq 0$ \cite{GV,CW} and if $s<0$ the equation is ill-posed in $H^{s}$ \cite{KPV,CCT}. Note however that the threshold obtained by the rescaled $\lambda \psi(\lambda^2t,\lambda x)$ solutions is $\dot H^{-\frac 12}$. Well-posedness was then proved to hold in  \cite{VaV,Gr} for data whose Fourier transform is in some $L^p$ spaces with $p<\infty$. From this point of view a natural space would be to consider  initial data with Fourier transform in $L^\infty$, 
\begin{equation}\label{spaceF}
\widehat{\psi_0}\in L^\infty,
\end{equation}
as this space is scaling invariant. In \cite{C} a result about the existence of a solution that belongs to almost this class is proved. More recently, several papers have appeared on uniform estimates in the Sobolev class, see \cite{KM} and \cite {KT} (and also \cite{KVZ}).

These results miss the case $\delta_{x=0}$, critical for the scaling, which is an important example for several reasons. One of them is that is also invariant under Galilean transformations, property that turns out to be crucial when \eqref{NLS0} is obtained from the equation of the binormal flow (see \eqref{BF}) thanks to the so-called Hasimoto transformation. By doing so, the gauge invariance of the geometric PDE has to be considered in \eqref{NLS0}. This amounts to study the more general equation
\begin{equation}\label{NLS00}
\left\{\begin{array}{cc}i\psi_t+\psi_{xx}\pm(|\psi|^2- A(t))\psi=0,\qquad A(t)\in \Bbb R,\\\psi(0)=\psi_0.\end{array}\right.
\end{equation}

At the level of Dirac measures as initial condition,  we have an explicit ill-posedness result of \eqref{NLS00} obtained first in \cite{BV1} and then improved in \cite{BV2} as follows. 
Let $a>0$ be small and let $0<\gamma<\frac 14$ and $ 0<\delta$. Consider the cubic Schr\"odinger equation
\begin{equation}\label{NLS}
i\psi_t+\psi_{xx}\pm(|\psi|^2-\frac {a^2}{t})\frac{\psi}{2}=0.
\end{equation}
First notice that the function
$$\psi_a(t,x)=a\frac{e^{i\frac{x^2}{4t}}}{\sqrt t}$$
is a solution with  $a\delta_0$ as initial data at time zero. Take $s\in \mathbb N$ and $0<\gamma<\frac 14$. 
We consider at time $t=1$ a perturbation $u(1)$ of $\psi_a(1)$, with $\partial_x^k u(1)$ small in the space 
\begin{equation}\label{Xtau}
X^\gamma=\{ f, \|f\|_{L^2}+\||\xi|^{2\gamma}\hat{f}(\xi)\|_{L^\infty(\xi^2\leq 1)}<+\infty\},
\end{equation}
for all $0\leq k\leq s$. 
Then, it is proved that \eqref{NLS} with $\psi_a(1)+u(1)$ as initial data at $t=1$ has a unique solution on $(0,1]$ that writes
\begin{equation}\label{psiu}\psi(t,x)=\frac{e^{i\frac{x^2}{4t}}}{\sqrt t}(a+\overline u)\left(\frac 1t,\frac xt\right),\end{equation}
with $\partial_x^k u(1)$ small in the space  
\begin{equation}\label{Ytau}
Y^{\gamma,\delta}=\left\{g, \,\sup_{t\geq 1}\,\left(\|g(t)\|_{L^2}+\frac{1}{t^\delta}\||\xi|^{2\gamma}\hat{g}(t,\xi)\|_{L^\infty(\xi^2\leq 1)}\right)<+\infty\right\}\cap L^4((1,\infty)L^\infty),
\end{equation}
for all $0\leq k\leq s$. 
Moreover, $u(t)$ scatters, in the sense that there exists a final state $f_+\in H^s$ such that $u(t)$ behaves for large time as $e^{\pm i a^2\log \sqrt{t}}e^{i(t-1)\partial_x^2}f_+$ in $L^\infty((t,\infty)H^s)$ norm.  In particular we obtain, imposing moreover that $u(1)$ belongs to weighted spaces, that
$$\left\|\psi(t,x)-a\frac{e^{i\frac{x^2}{4t}}}{\sqrt{t}}-\frac{e^{\pm ia^2\log \sqrt{t}}}{\sqrt{4\pi i}}\hat {\overline{f}}_+\left(-\frac x2\right)\right\|_{L^2}\leq C t^\frac 14.$$
This shows that as $t$ goes to zero the perturbative solution $\psi(t)$ behaves like $a\frac{e^{i\frac{x^2}{4t}}}{\sqrt t}$, that goes to $a\delta$ as $t$ approaches zero, and moreover that there is no limit in $L^2$ for $\psi(t)-a\frac{e^{i\frac{x^2}{4t}}}{\sqrt t}$. \\

\subsection{Continuation after the singularity is formed using the geometric structure}\label{sectgeom}
While at the Schr\"odinger level things are out of control at time zero as we have just seen, there is a way of having a new insight thanks to the geometric structure of the equation. 
In fact, the instability phenomena have been removed at the level of the Schr\"odinger map in \cite{BV4}  and in the following sense. 

Starting from $\psi$ the solution of \eqref{NLS} obtained above with the focusing non-linearity (i.e. take the $+$ sign in \eqref{NLS}), we construct a map $T:(0,1]\times\mathbb R\rightarrow \mathbb S^2$, together with $N:(0,1]\times\mathbb R\rightarrow \mathbb S^2+i\mathbb S^2$ by imposing $(T,\Re N,\Im N)(1,0)$ to be an orthonormal basis of $\mathbb R^3$, and the space and time evolution laws to be
\begin{equation}\label{syst}
\left\{\begin{array}{c}T_x(t,x)=\Re\overline{\psi}N(t,x),\\N_x(t,x)=-\psi T(t,x),\\
T_t=\Im\overline{\psi_x}N,\\
N_t=-i\psi_x T-i\,\frac{a^2-t|\psi|^2}{2t}\, N.
\end{array}\right.
\end{equation}
Then, it turns out that $T$ satisfies for $t\in (0,1]$ the equation of the Schr\"odinger map onto $\mathbb S^2$:
\begin{equation}\label{schmap}
T_t=T\land T_{xx}.
\end{equation}
Moreover, as $t$ goes to zero the vector $T(t,x)$ has a pointwise limit:
$$\exists T(0,x)= \underset{t\rightarrow 0}\lim T(t,x), \forall x\neq 0.$$
In the same way the modulated normal vector
\begin{equation}\label{deftildeN} \tilde N(t,x)=N(t,x)e^{i\Phi(t,x)}\,\,\,,\,\,\, \Phi(t,x)=a^2\log\frac{ |x|}{\sqrt{t}},\end{equation}
also has a pointwise limit as $t$ goes to zero 
$$\exists \tilde N(0,x)=\underset{t\rightarrow 0}\lim \tilde N(t,x), \forall x\neq 0.$$
So that at $t=0$ the limits $T(0)$ and $\tilde N(0)$ satisfy the following systems:
\begin{equation}\label{TN0}\left\{\begin{array}{c}T_x(0,x)=\Re\left(\widehat{f_+}\left(\frac x2\right)\,e^{-ia^2\log |x|}\,\tilde N(0,x)\right),\\ \,\\
\tilde N_x(0,x)=-\overline{\widehat{f_+}\left(\frac x2\right)\,e^{-ia^2\log |x|}}\,T(0,x),
\end{array}\right.\end{equation}
for $x\in(0,\pm\infty)$, with the values of $x=0^\pm$ given by a couple of vectors that depend on $a$ in a precise way . 

Finally, we proved in \cite{BV4} that there is a unique (in the sense that at the level of \eqref{NLS} the perturbations $u$ live in the $Y^{\gamma,\delta}$ spaces) nontrivial way of continuing the solution $T$ of \eqref{schmap} after time $t=0$. More precisely, having in mind the time invariance of the Schr\"odinger map \eqref{schmap}, it is enough to construct solutions $T^*$ of \eqref{schmap} for positive times, with initial data $T^*(0,x)=-T(0,-x)$. It turns out  that $T(0,0^+)\neq T(0,0^-)$, so that $T^*(0)$ is not a standard initial data for the Schr\"odinger map \eqref{schmap}. The first step to overcome this difficulty is to construct $(T^*,\tilde N^*)$ at time $t=0$ as solutions of \eqref{TN0}. Then, from the function appearing in this system to construct a new final state $f_+^*$, and thanks to the existence of the wave operators for the equation that $ u$ given in \eqref{psiu} satisfies, to construct a new function $u^*(t)$ on $(1,\infty)$ with this final state, yielding for $t\in(0,1]$ a new solution $\psi^*(t)$. Finally, via \eqref{syst} a frame $(T^*,\tilde N^*)$ is obtained for $t\in(0,1]$, with $T^*$ solution of \eqref{schmap}, a trace of it is obtained at time zero, which by a rigidity argument it is shown to be exactly $T^*(0)$. \\

\subsection{Lack of smoothness of some natural quantities associated to the geometric solution}
In this section we will exhibit some natural quantities associated to the solution of the Schr\"odinger map, whose construction we sketched in the previous section, that are non-smooth in time. 
More precisely, using the previous notations, we prove the following result.

\begin{theorem}\label{th}In view of the equations \eqref{syst} and the natural space \eqref{spaceF} we consider the quantities $\widehat{T_x}(t)$ and $\widehat{N_x}(t)$. The following strict inequalities hold:
\begin{equation}\label{Tdisc}
\underset{0<t\leq 1}{\inf}\|\widehat{T_x}(t)\|_{L^\infty}> \|\widehat{T_x}(0)\|_{L^\infty},
\end{equation}
and
\begin{equation}\label{Ndisc}
\underset{0<t\leq 1}{\inf}\|\widehat{N_x}(t)\|_{L^\infty}> \|\widehat{\tilde N_x}(0)\|_{L^\infty}.
\end{equation}
\end{theorem}

\subsection{Link with fluid mechanics phenomena} 
The present note was inspired from the following problem. We consider the binormal flow equation
\begin{equation}\label{BF}
\chi_t=\chi_x\wedge \chi_{xx},
\end{equation}
where $\chi(t)$ is a curve in $\mathbb R^3$ parametrized by the arclength parameter $x$ varying in $\mathbb R$ and in $[0,2\pi]$ for closed curves. This equation was derived by Da Rios in 1905 as a model for the evolution of a vortex filament in a 3-D fluid governed by Euler equations (\cite{DaR}). It is the simplest and the most used model for this kind of dynamics. We refer to the recent article \cite{JSe} about the validity of this approximation. 

Recent numerical simulations (\cite{DHV},\cite{JS})  done for the evolution \eqref{BF} taking polygons as initial condition, suggest some striking similarities with some phenomena observed for non-circular jets (see for instance Figure 6 of \cite{GG} and Figure 10 of \cite{GGP}). At the qualitative level two relevant facts are observed: 
\begin{itemize}
\item axis switching phenomena occurs,
\item  symmetries that are a multiple of the starting symmetry appear. 
\end{itemize}

Furthermore, in \cite{DHV} the authors show that these phenomena also appear in  \eqref{BF}. More precisely, if  the initial data $\chi(0)$ is a regular planar polygon with $M$ sides, supposing that a unique solution $\chi(t)$ exists, and integrating the Frenet system, it is proved that $\chi(t)$  at times 
$$t_{p,q}=(2\pi/M^2)p/q,$$
is a skew-polygon with a number of sides that is either $Mq$ or $Mq/2$. In particular in half a period (i.e. $t_{p,q}=\pi/M^2$) a regular planar  M-polygon reappears with the axis switched by an angle $2\pi/M$. This effect is a non--linear version of the so-called Talbot effect, (see  \cite{O}, \cite{ET}, and \cite{Ve} for the Talbot effect in non-linear settings ). 

Some more recent numerical simulations \cite{DHV2} show that the dynamics at time $0^+$ of any of the corners of the initial regular polygon is the one of the self-similar solution of \eqref{BF} that is determined by the angle and location of the corner. As a consequence, the dynamics at $0^+$ can be understood as the non-linear interaction of infinitely many filaments (as $q$ goes to infinity), one for each corner, that for infinitesimal times each resembles the one of the self-similar solution of \eqref{BF} studied in \cite{GRV}. We recall that the self-similar solutions of the binormal flow \eqref{BF} form a family of solutions $(\chi_a)_{a\in(0,\infty)}$ with a smooth profile at time $t=1$, and a limit at time $t=0$ that is precisely a corner described by $\chi(0,x)=xA_a^\pm$ for $x\in(0,\pm\infty)$, with $A^\pm$ unitary vectors determined in terms of the parameter $a$. 

Therefore, it seems rather natural to know up to what extent some quantities associated to the solutions of \eqref{BF} behave in a non-linear way. In this paper we start this study in the more accessible situation of the real line and more concretely we look at the solutions, obtained in our previous work \cite{BV4}, that are small perturbations of the self-similar solutions. We will mainly focus in the following aspects, related to fluid mechanics phenomena:
\begin{itemize}
  \item [(i)] Continuation after the singularity has been formed.
  \item [(ii)] Behavior of some conservation laws.
  \item [(iii)] Transfer of energy: Lack of continuity of some appropriate norm.
\end{itemize}

The first question was already answered in \cite{BV4}, and the continuation process has been already described, at the level of the tangent vector, in \S \ref{sectgeom}. 

For the second question we start making a remark on the so-called fluid impulse. We first recall, see \cite{MB} p. 24, that for a 3-D fluid governed by Euler equations, of vorticity $\omega$ regular and decaying at infinity, the fluid impulse 
$$\int_{\mathbb R^3}x\wedge\omega(t,x)dx$$
is conserved in time. In the case of the self-similar solutions $\chi_a$ of \eqref{BF}, as the vorticity is supposed to be concentrated along the curve $\chi_a(t)$, the corresponding quantity is the so-called linear momentum \cite{R}
$$\int_{\mathbb R} \chi_a(t,x)\wedge (\chi_a)_x(t,x)dx.$$
It is worth noting that explicit computations (see the Appendix) give
$$\int_{\mathbb R} \chi_a(t,x)\wedge (\chi_a)_x(t,x)dx=|t|(A_a^+-A_a^-).$$
In particular the linear momentum is not conserved, due of course to the behavior of the filament at  infinite, so that its modulus decays and eventually vanishes at the singularity formation time, and then grows again instantaneously after passing $t=0$. This means that close to a corner there is a neat transfer of linear momentum. This fact has been confirmed in the numerical experiments done in \cite{DHV2} for regular polygons. Also, and due to the fact that the number of corners depends on the rationality of the time, this local transfer of  momentum has a characteristic intermittent behavior.

For treating question (iii) we have to establish the functional spaces that are appropriate for our purposes. Theorem \ref{th} is an answer in this direction.

\section{Proof of Theorem \ref{th}}

\begin{proof}
We recall from \S 3.3 of \cite{BV4} that the system \eqref{TN0} for $x\in(0,\pm\infty)$, has initial data
\begin{equation}\label{TN0id}T(0,0^\pm)=RA_a^\pm\quad,\quad \tilde N(0,0^\pm)=RB_a^\pm,\end{equation}
for some rotation $R$ and $A_a^\pm,B_a^\pm$ are determined by the self-similar profiles of the binormal flow \eqref{BF}. Therefore
$$T_x(0,x)=R(A_a^--A_a^+)\delta_0+\Re\left(\hat f_+(\frac x2)e^{-ia^2\log |x|}\tilde N(0,x)\right),$$
and we have  
$$\left|\widehat{T_x}(0,\xi)-R(A_a^--A_a^+)\right|\leq C\int|\hat f_+(\frac x2)| dx\leq C\|f_+\|_{H^1}.$$
Since from \cite{GRV} we know that 
$$|A_a^--A_a^+|^2=4(1-A_{a,1}^2)=4(1-e^{-\pi a^2}),$$
we get
$$\left| |\widehat{T_x}(0,\xi)|-2\sqrt{1-e^{-\pi a^2}}\right|\leq C\|f_+\|_{H^1}.$$
As $\pi a^2 >1-e^{-\pi a^2}$, if $\|f_+\|_{H^1}$ is small enough, that in turn is obtained if $u(1)$ is small enough, we obtain 
\begin{equation}\label{estT0}2\sqrt{\pi}a> \|\widehat{T_x}(0)\|_{L^\infty}.\end{equation}

We shall get the lower bound $2\sqrt{\pi}a$ for $\|\widehat{T_x}(t)\|_{L^\infty}$ by looking at large frequencies.   From \eqref{syst} we have 
$$\widehat{T_x}(t,\xi)=\widehat {\Re\overline\psi N}(t,\xi).$$
We recall now from \cite{BV4} that the modulated vectors defined in \eqref{deftildeN} have a limit at $x$ large, independent of time:
$$\exists \lim_{x\rightarrow\pm\infty} \tilde N(t,x)=N^{\pm\infty},\,\Re N^{\pm\infty},\Im N^{\pm\infty}\in \mathbb S^2.$$
By using also the link \eqref{psiu} between $\psi$ and $u$, we can then write
\begin{equation}\label{That}\widehat{T_x}(t,\xi)=\int e^{-ix\xi}\,\Re\left(\frac{e^{-i\frac{x^2}{4t}}}{\sqrt t}(a+ u)\left(\frac 1t,\frac xt\right) e^{-i\phi(t,x)}\,\left(N^{+\infty} -g_N(t,x)\right)\right) dx,\end{equation}
with the function $g_N(t,x)$ defined by $g_N(t,x):=N^{+\infty}-\tilde{N}(t,x)$ satisfying
\begin{equation}\label{gas}g_N(t)\in L^\infty,\,g_N(t,x)\overset{x\rightarrow +\infty}{\longrightarrow}0.\end{equation}

The leading term in \eqref{That} is the same one as for the self-similar solutions, with $N^{+\infty}$ instead of $B^+$, so computations on it go the same. Indeed, from Lemma  \ref{oscint} iii) we prove below we get
$$\underset{|\xi|\rightarrow \infty}{\lim}\left|\int e^{-ix\xi}\,\Re\left(\frac{e^{-i\frac{x^2}{4t}}}{\sqrt t} ae^{-ia^2\log\frac{ |x|}{\sqrt{t}}}N^{+\infty}\right)dx-2\sqrt{\pi}a\Re \left(e^{i\xi^2 t-ia^2\log 2|\xi|\sqrt{t}-i\frac \pi 4}N^{+\infty}\right)\right|=0.$$
Then we notice the following orthogonality relation. By construction $\Re N(t,x) \perp \Im N(t,x)$, that writes
$$\Re\tilde{N}(t,x) e^{-i\phi(t,x)} \perp Im\tilde{N}(t,x) e^{-i\phi(t,x)},$$
and implies
$$\Re\tilde{N}(t,x)  \perp Im\tilde{N}(t,x).$$
From this together with \eqref{gas} it follows that
\begin{equation}\label{perpN}\Re N^{+\infty} \perp \Im N^{+\infty}.
\end{equation}
Using the orthogonality relation  \eqref{perpN} we have 
$$\left|\Re \left(e^{i\xi^2 t-ia^2\log 2|\xi|\sqrt{t}-i\frac \pi 4}N^{+\infty}\right)\right|=1,$$
and therefore
$$\underset{|\xi|\rightarrow \infty}{\lim}\left|\left|\int e^{-ix\xi}\,\Re\left(\frac{e^{-i\frac{x^2}{4t}}}{\sqrt t}a e^{-i\phi(t,x)}\,N^{+\infty} \right) dx\right|-2\sqrt{\pi}a\right|=0.$$
The three terms remaining to estimate in \eqref{That} are
$$\int e^{-ix\xi}\,\frac{e^{-i\frac{x^2}{4t}}}{\sqrt t} e^{-i\phi(t,x)}\,m(t,x)dx,\quad m(t,x)\in\{u\left(\frac 1t,\frac xt\right) ,g_N(t,x),u\left(\frac 1t,\frac xt\right) g_N(t,x)\}.$$
From Lemma \ref{oscint} i)-ii) these integrals tend to zero as $\xi$ goes to $-\infty$. So we have obtained\footnote{the same is valid with $\xi\rightarrow +\infty$ by working with $N^{-\infty}$ in \eqref{That}}
$$\underset{\xi\rightarrow -\infty}{\lim}|\widehat{T_x}(t,\xi)|=2\sqrt{\pi}a,$$
and in particular, in view of \eqref{estT0} we obtained the first inequality \eqref{Tdisc} of the Theorem:
$$\|\widehat{T_x}(t)\|_{L^\infty}> \|\widehat{T_x}(0)\|_{L^\infty}.$$

Concerning the normal vectors, similar computations can be done as follows. 
From system \eqref{TN0} with initial data \eqref{TN0id}  it yields that
$$\tilde N_x(0,x)=R(B_a^--B_a^+)\delta_0-\overline{\hat f_+(\frac x2)e^{-ia^2\log |x|}}\tilde N(0,x),$$
so we have  
$$\left|\widehat{\tilde N_x}(0,\xi)-R(B_a^--B_a^+)\right|\leq C\int|\hat f_+(\frac x2)| dx\leq C\|f_+\|_{H^1}.$$
Recalling that by symmetry of the profile of the self-similar solutions of \eqref{BF} the first coordinate of its normal and binormal vectors are odd, and that its second and third coordinates are even, we have 
$$|B_a^--B_a^+|=2|B_{a,1}|,$$
and we get
$$\left| |\widehat{\tilde N_x}(0,\xi)|-2|B_{a,1}|\right|\leq C\|f_+\|_{H^1}.$$
From the system \eqref{syst} on the first coordinate we obtain the conservation law $T_1(t,x)^2+|N_1(t,x)|^2=1$, so we also have $T_1(t,x)^2+|\tilde N_1(t,x)|^2=1$, and letting $x$ go to $+\infty$ in the case of $\chi_a$ we get that  $4|B_{a,1}|^2=4(1-A_{a,1}^2)=4(1-e^{-\pi a^2})$. As $4\pi a^2 >4(1-e^{-\pi a^2})$, if $\|f_+\|_{H^1}$ is small enough we obtain
\begin{equation}\label{estN0}2\sqrt{\pi}a> \|\widehat{\tilde N_x}(0)\|_{L^\infty}.\end{equation}
Again we shall get the lower bound $2\sqrt{\pi}a$ for $\|\widehat{N_x}(t)\|_{L^\infty}$ by looking at large frequencies.   From \eqref{syst} we have 
$$\widehat{N_x}(t,\xi)=-\widehat {\psi T}(t,\xi).$$
Recall from \cite{BV4} that 
$$\exists \lim_{x\rightarrow\pm\infty} T(t,x)=T^{\pm\infty}\in \mathbb S^2.$$
We can then write
\begin{equation}\label{Nhat}\widehat{N_x}(t,\xi)=-\int e^{-ix\xi}\,\frac{e^{i\frac{x^2}{4t}}}{\sqrt t}(a+ \overline{u})\left(\frac 1t,\frac xt\right) \left(T^{+\infty}-g_T(t,x)\right) dx,\end{equation}
with the function $g_T(t,x)$ defined by $g_T(t,x):=T^{+\infty}-T(t,x)$ satisfying
\begin{equation}\label{gasbis}g_T(t)\in L^\infty,\,g_T(t,x)\overset{x\rightarrow +\infty}{\longrightarrow}0.\end{equation}
Lemma \ref{oscint} insures us that
$$\underset{|\xi|\rightarrow \infty}{\lim}\left|\int e^{-ix\xi}\frac{e^{i\frac{x^2}{4t}}}{\sqrt t} aT^{+\infty}dx-2\sqrt{\pi}ae^{i\xi^2 t-i\frac \pi 4}T^{+\infty}\right|=0,$$
and that the three following terms are of order $o(\xi)$:
$$\int e^{-ix\xi}\,\frac{e^{i\frac{x^2}{4t}}}{\sqrt t}\,m(t,x)dx,\quad m(t,x)\in\{\overline{u}\left(\frac 1t,\frac xt\right) ,g_T(t,x),\overline{u}\left(\frac 1t,\frac xt\right) g_T(t,x)\}.$$
Therefore from \eqref{Nhat} we obtain 
$$\underset{\xi\rightarrow -\infty}{\lim}|\widehat{N_x}(t,\xi)|=2\sqrt{\pi} a,$$
and in view of \eqref{estN0} the inequality \eqref{Ndisc} follows. 

Summarizing, the proof of the theorem is achieved once Lemma \ref{oscint} is proved. 
\end{proof}

\begin{lemma}\label{oscint} Let $t\in(0,1)$. We estimate the oscillatory integrals
$$I_\xi=\int e^{-ix\xi}\,\frac{e^{-i\frac{x^2}{4t}}}{\sqrt t} e^{-ia^2\log\frac{ |x|}{\sqrt{t}}}m(t,x)dx,\quad\tilde I_\xi=\int e^{-ix\xi}\,\frac{e^{-i\frac{x^2}{4t}}}{\sqrt t} m(t,x)dx,$$
in the following cases. \\
i) If $m(t)\in H^3$ we have
\begin{equation}\label{limitgen}\underset{\xi\rightarrow -\infty}{\lim}I_\xi =0,
\end{equation}
\begin{equation}\label{limitgenbis}\underset{\xi\rightarrow -\infty}{\lim}\tilde I_\xi =0,
\end{equation}
Note that this assumption is satisfied by $m(t,x)=u\left(\frac 1t,\frac xt\right)$ if $u\in H^3$.\\
ii) The convergence \eqref{limitgen} is also valid for $m(t,x)=g_T(t,x)=\tilde N(t,x)-N^{+\infty}$, that is not in $H^1$, and also for $m(t,x)=u\left(\frac 1t,\frac xt\right)g_T(t,x)$. The convergence \eqref{limitgenbis} is also valid for $m(t,x)=g_T(t,x)=\tilde T(t,x)-T^{+\infty}$, and for $m(t,x)=u\left(\frac 1t,\frac xt\right)g_T(t,x)$.\\
iii) If $m(t,x)=1$, we have
\begin{equation}\label{limit}\underset{\xi\rightarrow -\infty}{\lim}\left|I_\xi-2e^{i\xi^2 t}e^{-ia^2\log 2|\xi|\sqrt{t}}\int e^{-is^2}ds\right|=0,\,\underset{\xi\rightarrow -\infty}{\lim}\left|\tilde I_\xi-2e^{i\xi^2 t}\int e^{-is^2}ds\right|=0.\end{equation}
\end{lemma}
\begin{proof}
Let us start with the integrals $I_\xi$. 
We shall first treat i) in a manner that will turn out to be valid also to the cases ii), up to  the estimate of a last oscillatory integral that will be treated case by case.  Finally we shall point out how the case iii) can be treated similarly to the case i).\\

First we get rid of the time by a change of variable:
$$I_\xi:=\int e^{-ix\xi}\,\frac{e^{-i\frac{x^2}{4t}}}{\sqrt t} e^{-ia^2\log\frac{ |x|}{\sqrt{t}}}m(t,x)dx=2\int e^{-i2\sqrt{t} y \xi}e^{-iy^2}e^{-ia^2\log 2|y|}m(t,2\sqrt{t}y)dy.$$

Now we notice that the part near the origin is negligible. More precisely, let $\chi$ be a smooth cuttoff such that $\chi(s)=0$ for $|s|\leq 1$ and $\chi(s)=1$ for $|s|\geq 2$. We split the integral into two corresponding pieces. We re-split the first one into two regions:
$$I_{\xi,1}:= 2\int_{|y|<\frac 1\xi} e^{-i2\sqrt{t} y \xi}e^{-iy^2}e^{-ia^2\log 2|y|}m(t,2\sqrt{t}y)(1-\chi(y))dy$$
$$+2\int_{\frac 1\xi\leq |y|\leq 2} e^{-i2\sqrt{t} y \xi}e^{-iy^2}e^{-ia^2\log 2|y|}m(t,2\sqrt{t}y)(1-\chi(y))dy.$$
On the first region we get convergence to zero as $|\xi|\rightarrow\infty$, with the rate of decay $\frac 1{|\xi|}$, just by using the fact that $m(t)\in L^\infty$. On the second region we integrate by parts from $e^{-i2\sqrt{t} y \xi}$ and the worse term gives us a rate of convergence to zero by $\frac{\log|\xi|}{|\xi|}$, provided that $m(t),m'(t)\in L^\infty$.\\

We are left with
$$I_{\xi,2}:=2e^{i\xi^2 t}\int e^{-i(y+\xi\sqrt{t})^2}e^{-ia^2\log 2|y|}m(t,\sqrt{t}y)\chi(y)dy$$
$$=2e^{i\xi^2 t}\int e^{-is^2}e^{-ia^2\log 2|s-\xi\sqrt{t}|}m(t,2\sqrt{t}(s-\xi\sqrt{t}))\chi(s-\xi\sqrt{t})dy,$$
so we have to show by replacing $-\sqrt{t}\xi$ by $\eta$ that
\begin{equation}\label{diffest}\underset{\eta\rightarrow +\infty}{\lim} J_\eta:=\underset{\eta\rightarrow +\infty}{\lim}\int e^{-is^2}f(s+\eta) ds=0,\end{equation}
for $f(r)=e^{-ia^2\log 2|r|}m(t,2\sqrt{t}r)\,\chi(r)$. 
%Notice that \eqref{limit} follows as $\chi(\eta)\overset{\eta\rightarrow\infty}{\longrightarrow}1$. The first part of the statement \eqref{limitgen} follows from the hypothesis $m(t,x)\overset{x\rightarrow\infty}{\longrightarrow}0$.
We shall use frequently the fact that for  $|s|<\frac{|\eta|}2$ we have, as by hypothesis $m(t,x)\overset{x\rightarrow\infty}{\longrightarrow}0$, the decay
\begin{equation}\label{diffm}f(s+\eta)=o(\frac1{|\eta|})\end{equation}
%We shall use frequently the fact that for  $|s|<\frac{|\eta|}2$ we have the existence of $|s_0|<\frac{|\eta|}2$ such that, for $|\eta|>4$,
%\begin{equation}\label{diff}|f(s+\eta)-f(\eta)|=|s| |\frac{-ia^2}{2|s_0+\eta|}m(t,2\sqrt{t}s_0)\chi(s_0+\eta)+m(t,2\sqrt{t}s)\chi'(s_0+\eta)+2\sqrt{t}m'(t,2\sqrt{t}s)\chi(s_0+\eta)|\end{equation} $$\leq C \frac{|s|}{|s_0+\eta|}\leq C\frac{|s|}{|\eta|}.$$
We split the integral $J_\eta$ into two pieces, using again the localization $\chi$:
$$J_\eta=\int e^{-is^2}f(s+\eta) (1-\chi(s))ds+\int e^{-is^2}f(s+\eta)\chi(s)ds:=J_{\eta,1}+J_{\eta,2}.$$
The convergence of $J_{\eta,1}$ is insured by \eqref{diffm} and by the fact that the support of $1-\chi$ is bounded. For the $J_{\eta,2}$ we use integrations by parts:
$$J_{\eta,2}=\int e^{-is^2}f(s+\eta)\chi(s)ds=-\int e^{-is^2}\frac{i}{2s^2}f(s+\eta)\chi(s)ds$$
$$+\int e^{-is^2}\frac i{2s}f(s+\eta) \chi'(s)ds+\int e^{-is^2}\frac i{2s}f'(s+\eta) \chi(s)ds:=J_{\eta,2}^1+J_{\eta,2}^2+J_{\eta,2}^3.$$

The first integral $J_{\eta,2}^1$ restricted to $1<|s|<\frac{|\eta|}2$ satisfies the convergence to zero, by using again \eqref{diffm}. On its remaining part $|s|\geq \frac{|\eta|}2$ we get a $\frac 1{|\eta|}$ decay just by using $m(t)\in L^\infty$ and integrating $\frac 1{s^2}$. 

The second integral $J_{\eta,2}^2$ is restricted on $1<|s|<2$, so  it converges again thanks to \eqref{diffm}. 

The third integral $J_{\eta,2}^3$, which lives on $|s|>1, |s+\eta|>1$, can be upper-bounded by
$$\left|\int e^{-is^2}\frac i{2s}f'(s+\eta) \chi(s)ds\right|\leq C\int\left|\frac{\chi(s+\eta)\chi(s)}{s(s+\eta)} \right|ds+C\int\left|\frac{\chi'(s+\eta)\chi(s)}{s}\right|ds$$
$$+C\left|\int e^{-is^2}\frac{e^{-ia^2\log 2|s+\eta|}m'(t,2\sqrt{t}(s+\eta))\chi(s+\eta)\chi(s)}{s}ds\right|.$$

For the first integral, 
%either we keep the oscillating factor and get, after an extra integration by parts, a decay of $\frac{\log|\eta|}{|\eta|}$, or we proceed the following way. 
on the region $1<|s|<|\eta|-\sqrt{|\eta|}$ we get a $\frac{\log|\eta|}{\sqrt{|\eta|}}$ decay, on the region $|\eta|-\sqrt{|\eta|}<|s|<|\eta|+\sqrt{|\eta|}$ we get a $\frac{1}{\sqrt{|\eta|}}$ decay by using $ |s+\eta|>1$, and on the region $|\eta|+\sqrt{|\eta|}<|s|$ we get a $\frac{1}{|\eta|}$ decay. 

The second integral lives on $ 1<|s+\eta|<2$ so in particular $\frac 1{|s|}\leq \frac C{|\eta|}$ and we obtain a  convergence with a  $\frac 1{|\eta|}$ rate of decay. \\

So we are left with showing the convergence to zero of the third integral
$$K_\eta=\int e^{-is^2}\frac{e^{-ia^2\log 2|s+\eta|}m'(t,2\sqrt{t}(s+\eta))\chi(s+\eta)\chi(s)}{s}ds.$$

Note that until now all estimates, except the one on $I_{\xi,1}$ that uses also $m'\in L^\infty$, used only the properties
$$m(t)\in L^\infty, m(t,x)\overset{x\rightarrow +\infty}{\longrightarrow}0.$$
These conditions are satisfied also for the cases $m(t,x)=g_N(t,x)=\tilde N(t,x)-N^\infty$, and for $m(t,x)=u\left(\frac 1t,\frac xt\right)g_N(t,x)$. 
%Indeed, as the normal vectors are of norm $2$, we have $g(t)\in L^\infty$. Finally, from the asymptotics of $g(t)$ at infinity we obtain the convergence $g(t,x)\overset{x\rightarrow\infty}{\longrightarrow}0$.  
Estimating $I_{\xi,1}$ with $g_N(t,x)$ or with $u\left(\frac 1t,\frac xt\right)g(t,x)$ can be done the same as above, and the same rate of convergence to zero by $\frac{\log|\xi|}{|\xi|}$ is recovered by using the fact that $g_N'(t)=\tilde N'(t)=(-\psi(t) T(t)+\frac{ia^2}{s}N)e^{i\Phi}$ is integrable on $\frac 1\xi\leq |s|\leq 2$.\\
%by splitting it into $|y|<\frac 1{\sqrt{|\xi|}}$ and $\frac 1{\sqrt{|\xi|}}\leq|y|\leq 2$.\\
%Note that the proof is already finished in the case $m(t,x)=1$. \\

In order to end the proof of i), we perform in $K_\eta$ an extra integration by parts, use $ |s+\eta|>1$ and get as an upper-bound
$$|K_\eta|\leq C\int_{|s|>1}\frac{|m'(t,2\sqrt{t}(s+\eta))|+|m''(t,2\sqrt{t}(s+\eta))|}{s^2}ds.$$
On $1<|s|<\frac{|\eta|}2$ we use the decay hypothesis on $m$, and on $\frac{|\eta|}2<|s|$ we get a convergence of rate $\frac 1{|\eta|}$.\\

For estimating $K_\eta$ in the case $m(t,x)=g_N(t,x)$ of ii), since $g_N'(t)=\tilde N'(t)=(-\psi(t) T(t)+\frac{ia^2}{s}N)e^{i\Phi}$,  it is enough to estimate
$$\int e^{2is\eta}\frac{(a+\overline{u})\left(\frac 1t,\frac {2(s+\eta)}{\sqrt{t}}\right)T(t,2\sqrt{t}(s+\eta))\chi(s+\eta)\chi(s)}{s}ds$$
and
$$\int \left|\frac{\chi(s+\eta)\chi(s)}{s(s+\eta)}\right|ds.$$
The second term has been already treated. 
In the first term we perform an integration by parts and get an $\frac 1{|\eta|}$ rate of convergence easily up to when the derivative falls on $T$, which leads, as from \eqref{syst} we have $T_s=\Re \overline{\psi}N$, to
$$\frac 1\eta\int e^{2is\eta}\frac{e^{-ia^2\log 2|s+\eta|}(a+\overline{u})\left(\frac 1t,\frac {2(s+\eta)}{\sqrt{t}}\right)\Re \overline{\psi}N(t,2\sqrt{t}(s+\eta))\chi(s+\eta)\chi(s)}{s}ds$$
$$=\frac {e^{-i\eta^2}}{2\sqrt{t}\eta}\int e^{-is^2}\frac{e^{-ia^2\log 2|s+\eta|}|a+u|^2\left(\frac 1t,\frac {2(s+\eta)}{\sqrt{t}}\right)N(t,2\sqrt{t}(s+\eta))\chi(s+\eta)\chi(s)}{s}ds$$
$$+\frac {e^{-i3\eta^2}}{2\sqrt{t}\eta}\int e^{i(s+2\eta)^2}\frac{e^{-ia^2\log 2|s+\eta|}(a+\overline{u})\left(\frac 1t,\frac {2(s+\eta)}{\sqrt{t}}\right) \overline{N}(t,2\sqrt{t}(s+\eta))\chi(s+\eta)\chi(s)}{s}ds.$$
For obtaining the convergence to zero in $\eta$ for the first integral we redo an integration by parts from the quadratic phase. We do the same for the second integral, after removing first the region $s+2\eta\approx 0$ that gives a decay of type $\frac{\log\eta}{\eta}$.\\

Estimating $K_\eta$ in the case $m(t,x)=u\left(\frac 1t,\frac xt\right)g_N(t,x)$ of ii)  is a mix between the cases $m(t,x)=u\left(\frac 1t,\frac xt\right)$ and $m(t,x)=g_N(t,x)$ that do not cause new issues.\\

The proof of \eqref{limit} of the case $m=1$ in iii) goes the same as the proof of i) with $J_\eta$ replaced by
$$\int e^{-is^2}(f(s+\eta)-f(\eta)) ds,$$
and with \eqref{diffm} replaced by:
$$|f(s+\eta)-f(\eta)|\leq C\frac{|s|}{|\eta|}.$$
Indeed, this last inequality is valid as for  $|s|<\frac{|\eta|}2$ we have the existence of $|s_0|<\frac{|\eta|}2$ such that, by taking $|\eta|>4$,
$$|f(s+\eta)-f(\eta)|=|s||\frac{-ia^2}{2|s_0+\eta|}\chi(s_0+\eta)+\chi'(s_0+\eta)|\leq C \frac{|s|}{|s_0+\eta|}\leq C\frac{|s|}{|\eta|}.$$

Finally, the integrals $\tilde I_\xi$ can be treated similarly as $I_\xi$.

\end{proof}

\section{Appendix: computations on the linear momentum}

We consider here the  family of self-similar solutions $(\chi_a)_{a\in(0,\infty)}$ of the binormal flow \eqref{BF} described in \cite{GRV}. They are determined by their profiles $G_a$ in the sense that
$$\chi_a(t,x)=\sqrt{t}G_a(\frac{x}{\sqrt{t}}),$$
and their profiles satisfy the asymptotics
\begin{equation}\label{Gas}\underset{x\rightarrow\pm\infty}{\lim}G'_a(x)=A_a^\pm,\end{equation}
with $A^\pm$ unitary vectors determined in terms of the parameter $a$. From the binormal flow \eqref{BF} the following ordinary equation is obtained for $G_a$:
\begin{equation}\label{Gaeq}G_a(s)-sG'_a(s)=G'_a(s)\wedge G''_a(s).\end{equation}
%After differentiation we obtain \begin{equation}\label{Ga}-sG_a''=G'_a(s)\wedge G'''_a(s).\end{equation}
Finally, this solution is extended to $t<0$ as
$$\chi_a(t,x)=\chi_a(-t,-x)\qquad t<0.$$
Now for $t>0$ we compute:
$$\int \chi_a(t,x)\wedge (\chi_a)_x(t,x) dx=\int \sqrt{t} G_a(\frac x{\sqrt{t}})\wedge G_a'(\frac x{\sqrt{t}}) dx=t\int G_a(s)\wedge G'_a(s) ds.$$
By using \eqref{Gaeq} we get
$$\int \chi_a(t,x)\wedge (\chi_a)_x(t,x) dx=t\int (G'_a(s)\wedge G''_a(s))\wedge G'_a(s) ds=t\int G_a''(s) ds,$$
as $G_a$ is parametrized by arclength and $G_a'(s)$ is orthogonal to $G_a''(s)$. Therefore \eqref{Gas} allows us to conlude the identity \eqref{impulse} for $t>0$:
$$\int \chi_a(t,x)\wedge (\chi_a)_x(t,x) dx=t(A^+-A^-).$$ 
Moreover, since $\chi(0,x)=xA_a^\pm$ for $x\in(0,\pm\infty)$, we obtain this identity also at time $t=0$:
$$\int \chi_a(0,x)\wedge (\chi_a)_x(0,x) dx=0.$$ 
As a consequence we get:
\begin{equation}\label{impulse}\int \chi_a(t,x)\wedge (\chi_a)_x(t,x) dx=|t|(A^+-A^-), \quad t\in \Bbb R.
\end{equation}

\begin{remark}
If is not obvious whether or not an identity similar to \eqref{impulse} remains true  for small perturbations of the self-similar solution $\chi_a$. Already some conditions have to be made for the initial datum in order the momentum 
$$\int \chi(t,x)\wedge \chi_x(t,x) dx$$
to be well-defined. For instance, at initial time $t_0$ we could impose as datum a solution $\chi(t_0,x)$ of the equation 
$$G(s)-sG'(s)=(1+\epsilon(s))G'(s)\wedge G''(s),$$
to be solved on $[0,\pm\infty)$ with same initial data $G(0)\in\mathbb R^3$ and $G'(0)\in\mathbb S^2$, 
for a small regular function $\epsilon$ decaying to 0 at $\pm\infty$. Note that by taking the exterior product with $G'$ we get $G\wedge G'=(1+\epsilon)G''$, so by taking the scalar product with $G'$ we get the conservation of $|G'(s)|^2$. Then, proceeding as for the self-similar profiles $G_a$, see \cite{GRV},  one can compute the curvature and torsion and to conclude that $G''(s)$ has limits $G^\pm$ at $\pm\infty$. In particular we can compute as above
$$\int \chi(t_0,x)\wedge \chi_x(t_0,x) dx=\int G(x)\wedge G'(x)dx=\int(1+\epsilon(x))(G'(s)\wedge G''(s))\wedge G'(x)dx$$
$$=\int(1+\epsilon(x))G''(x)dx=G^+-G^--\int\epsilon'(x)G'(x)dx,$$
which is a finite quantity if for instance $\epsilon'$ is integrable. Another question is to see if the momentum is still well-defined for $t\neq t_0$ and behaves as in \eqref{impulse}. We plan to address these questions in the future.
\end{remark}

{\bf{Acknowledgements:}}  This research is supported by the ANR project "SchEq" ANR-12-JS01-0005-01, by ERCEA Advanced Grant 2014 669689 - HADE, by the MINECO projects MTM2014-53850-P and SEV-2013-0323. and  by the Basque Government projects IT-641-13 and  BERC 2014-2017.


\begin{thebibliography}{99999}
\bibitem{BV1}  V.~Banica and L.~Vega, 
On the stability of a singular vortex dynamics. 
{\it Comm. Math. Phys.} {\bf286} (2009), 593--627.
\bibitem{BV2}  V.~Banica and L.~Vega, 
{\it Scattering for 1D cubic NLS and singular vortex dynamics,} 
J. Eur. Math. Soc. {\bf 14} (2012), 209--253.
\bibitem{BV4}  V.~Banica and L.~Vega, 
{\it  The initial value problem for the binormal flow with rough data,} 
Ann. Sci. \'Ec. Norm. Sup\'er. {\bf 48} (2015), 1421--1453.
\bibitem{CW} 
T.~Cazenave and F.B.~Weissler,
{\it The Cauchy problem for the critical nonlinear Schrödinger equation,} 
Non. Anal. TMA {\bf 14} (1990), 807--836. 
\bibitem{C} {M.~Christ}, {\it Power series solution of a nonlinear Schrödinger equation}. Mathematical aspects of nonlinear dispersive equations, 131--155, Ann. of Math. Stud., {\bf163}, (2007) Princeton Univ. Press, Princeton, NJ.
\bibitem{CCT}
M.~Christ, J.~Colliander, and T.~Tao, 
{\it Ill-posedness for nonlinear
  {S}chr\"odinger and wave equations}, 
ArXiv:0311048.
\bibitem{DHV} F.~de la Hoz and L.~Vega, {\it Vortex filament equation for a regular polygon}. Nonlinearity {\bf27} (2014), 3031--3057.
\bibitem{DHV2} F.~de la Hoz and L.~Vega, {\it On the Relationship between the One-Corner Problem and the $M$-Corner Problem for the Vortex Filament Equation,} In preparation.
\bibitem{DaR} L. S. Da Rios,  
On the motion of an unbounded fluid with a vortex filament of any shape.  
{\em Rend. Circ. Mat. Palermo} {\bf22} (1906), 117--135.
\bibitem{ET} M.B.~Erdogan and N.~Tzirakis,
{\it Talbot effect for the cubic non-linear Schrödinger equation on the torus}, 
Math. Res. Lett. {\bf20} (2013), 1081--1090.
\bibitem{GV}
J.~Ginibre and G.~Velo, 
{\it On a class of nonlinear {S}chr\"odinger
  equations. {II} {S}cattering theory, general case}, 
  J. Funct. Anal. {\bf 32} (1979), 33--71.
  \bibitem{GG} 
  F.F~Grinstein and E.J.~Gutmark, {\it Flow control with noncircular jets} Annu. Rev. Fluid Mech. {\bf31} (1999), 239--272.
\bibitem{GGP} 
F.F~Grinstein, E.J.~Gutmark, and T.~Parr, {\it Near field dynamics of subsonic free square jets. A computational and experimental study}
 Physics of Fluids {\bf7} (1995), 1483--1497.
  \bibitem{Gr}
A.~Gr\"unrock, 
{\it Bi- and trilinear Schr\"odinger estimates in one space dimension with applications to cubic NLS and DNLS, }
  Int. Math. Res. Not.  
  {\bf 41} (2005), 2525--2558.
  \bibitem{GRV} S.~Guti\'errez, J.~Rivas, and L.~Vega, {\it Formation of singularities and self-similar vortex motion under the localized induction approximation} Commun. PDE {\bf28} (2003) 927--968. 
  \bibitem{JSe} 
  R. ~Jerrard, C.~Seis, {\it On the vortex filament conjecture for Euler flows}, ArXiv:1603.00227.
  \bibitem{JS}
  R. ~Jerrard and D~Smets, {\it On the motion of a curve by its binormal curvature,} J. Eur. Math. Soc. (JEMS) {\bf17} (2015), no. 6, 148--1515. 
  \bibitem{KM}
  T.~ Kappeler, J.C.~ Molnar, {\it
On the wellposedness of the defocusing mKdV equation below L2}, 
ArXiv:1606.07052.
  \bibitem{KPV}
C. Kenig, G. Ponce and L. Vega,
 {\em On the ill-posedness of some canonical non-linear dispersive equations, }
Duke Math. J.
{\bf 106}  (2001), 716--633.
 \bibitem{KT} H.~ Koch and  D.~Tataru, 
{\it Conserved energies for the cubic NLS in 1-d,} ArXiv:1607.02534. 
\bibitem{KVZ}R.~ Killip, M.~Visan Monica, and X.~Zhang,  Talk Bonn March 2016 and private communication, 2016.
\bibitem{MB}A.J.~Majda. A.L.~Bertozzi  {\it Vorticity and Incompressible Flows,} Cambridge Texts in Applied Mathematics. Cambridge U. Press, 2002.
\bibitem{O} P.J.~Olver,  {\it Dispersive quantization,} Am. Math. Monthly {\bf117} (2010) 599--610.
\bibitem{R} R.L. ~Ricca, {\it Physical interpretation of certain invariants for vortex filament motion under LIA,} Phys. Fluids A  {\bf 4} (1992), 938--944.
\bibitem{VaV}
A.~Vargas and L.~Vega,
 {\it Global well-posedness for 1d non-linear  Schr\"odinger equation for data with an infinite $L^2$ norm, }
J. Math. Pures Appl.
{\bf 80}  (2001), 1029--1044.
\bibitem{Ve}
 L.~Vega,
 {\it The dynamics of vortex filaments with corners,} Commun. Pure Appl. Anal. 14 (2015), no. 4, 1581--1601.
\end{thebibliography}
\end{document}